\documentclass{amsart}

\usepackage{amsmath,amssymb}
\usepackage{amsthm}
\usepackage{enumitem}
\usepackage{multicol}
\usepackage{url}

\usepackage{color}

\usepackage{tikz,tkz-graph}
\tikzset{
every node/.style={circle, draw, inner sep=2pt},
every picture/.style={thick}
}
\usetikzlibrary{matrix,arrows,calc}

\newtheorem{theorem}{Theorem}
\newtheorem{lemma}[theorem]{Lemma}
\newtheorem{claim}[theorem]{Claim}
\newtheorem{proposition}[theorem]{Proposition}
\newtheorem{corollary}[theorem]{Corollary}

\theoremstyle{definition}
\newtheorem{definition}[theorem]{Definition}
\newtheorem{observation}[theorem]{Observation}
\newtheorem{remark}[theorem]{Remark}
\newtheorem{example}[theorem]{Example}
\newtheorem{question}[theorem]{Question}

\newcommand{\diag}{\operatorname{diag}}
\newcommand{\minors}{\operatorname{minors}}

\newcommand{\Ga}{{\sf bull}}
\newcommand{\Gb}{{\sf dart}}
\newcommand{\Gc}{{\sf house}}
\newcommand{\Gd}{{\sf gem}}
\newcommand{\Ge}{{\sf full-house}}
\newcommand{\Gf}{{\sf G_{6,5}}}
\newcommand{\Gg}{{\sf 5-pan}}
\newcommand{\Gh}{{\sf G_{6,7}}}
\newcommand{\Gi}{{\sf G_{6,8}}}
\newcommand{\Gj}{{\sf G_{6,9}}}
\newcommand{\Gk}{{\sf G_{6,10}}}
\newcommand{\Gl}{{\sf co-twin-house}}
\newcommand{\Gm}{{\sf G_{6,12}}}
\newcommand{\Gn}{{\sf co\text{-}twin\text{-}C_5}}
\newcommand{\Go}{{\sf G_{6,14}}}
\newcommand{\Gp}{{\sf C_7}}
\newcommand{\Gq}{{\sf G_{6,15}}}

\begin{document}

\title{On graphs with 2 trivial distance ideals}
\author{Carlos A.  Alfaro}
\address{
Banco de M\'exico,
Mexico City,
Mexico
}
\email{alfaromontufar@gmail.com}
\thanks{Carlos A. Alfaro was partially supported by CONACyT and SNI}
\keywords{distance ideals, distance matrix, forbidden induced subgraph, Smith normal form.}
\subjclass[2000]{05C25, 05C50, 05E99, 13P15, 15A03, 68W30.}
\date{\today}

\maketitle

\begin{abstract}
Distance ideals generalize the Smith normal form of the distance matrix of a graph.
The family of graphs with 2 trivial distance ideals contains the family of graphs whose distance matrix has at most 2 invariant factors equal to 1.
Here we give an infinite family of forbidden induced subgraphs for the graphs with 2 trivial distance ideals.
These are also related with other well known graph classes.

\end{abstract}

\section{Introduction}
Through the paper all graphs will be considered connected.
Let $G=(V,E)$ be a connected graph and $X_G=\{x_u \, : \, u\in V(G)\}$ a set of indeterminates.
The {\it distance} $d_G(u,v)$ between the vertices $u$ and $v$ is the number of edges of a shortest path between them.
Let $\diag(X_G)$ denote the diagonal matrix with the indeterminates in the diagonal.
The {\it distance matrix} $D(G)$ of $G$ is the matrix with rows and columns indexed by the vertices of $G$ where the $uv$-entry is equal to $d_G(u,v)$.
Thus the {\it generalized distance matrix} $D(G,X_G)$ of $G$ is the matrix with rows and columns indexed by the vertices of $G$ defined as $\diag(X_G)+D(G)$.
Note we can recover the distance matrix from the generalized distance matrix by evaluating $X_G$ at the zero vector, that is, $D(G)=D(G,\bf{0})$.

Let $\mathcal{R}[X_G]$ be the polynomial ring over a commutative ring $\mathcal{R}$ in the variables $X_G$.
For all $i\in[n]:=\{1,..., n\}$, the $i${\it-th distance ideal} $I^\mathcal{R}_i(G,X_G)$ of $G$ is the ideal, over $\mathcal{R}[X_G]$, given by
$\langle \minors_i(D(G,X_G))\rangle$,
where $n$ is the number of vertices of $G$ and ${\rm minors}_i(D(G,X_G))$ is the set of the determinants of the $i\times i$ submatrices of $D(G,X_G)$.




Distance ideals were defined in \cite{at} as a generalization of the Smith normal form of distance matrix and the distance spectra of graphs.
Through this paper will be interested in the case when $\mathcal{R}$ is the ring of integer numbers.
For this reason we will omit $\mathcal{R}$ in the notation of distance ideal.

\subsection{Distance ideals and Smith normal form of distance matrix}

Smith normal forms have been useful in understanding algebraic properties of combinatorial objects see \cite{stanley}.
For instance, computing the Smith normal form of the adjacency or Laplacian matrix is a standard technique used to determine the Smith group and the critical group of a graph, see \cite{alfaval0,merino,rushanan}.

Smith normal forms can be computed using row and column operations.
In fact, Kannan and Bachem found in \cite{KB} polynomial algorithms for computing the Smith normal form of an integer matrix.
An alternative way of obtaining the Smith normal form is as follows.
Let $\Delta_i(G)$ denote the {\it greatest common divisor} of the $i$-minors of the distance matrix $D(G)$, then its $i$-{\it th} invariant factor, $f_i$, is equal to $\Delta_i(G)/ \Delta_{i-1}(G)$, where $\Delta_0(G)=1$.
Thus the Smith normal form of $D(G)$ is equal to $\diag(f_1,f_2, \dots, f_r,0, \dots,0)$.

Little is known about the Smith normal forms of distance matrices.
In \cite{HW}, the Smith normal forms of the distance matrices were determined for trees, wheels, cycles, and complements of cycles and are partially determined for complete multipartite graphs.
In \cite{BK}, the Smith normal form of the distance matrices of
unicyclic graphs and of the wheel graph with trees attached to each
vertex were obtained.

An ideal is said to be {\it unit} or {\it trivial} if it is equal to $\langle1\rangle$. 
Let $\Phi(G)$ denote the maximum integer $i$ for which $I^\mathbb{Z}_i(G,X_G)$ is trivial.
Let $\Lambda_k$ denote the family of graphs with at most $k$ trivial distance ideals over $\mathbb{Z}$.
Note that every graph with at least one non-loop edge has at least one trivial distance ideals.
On the other hand, let $\phi(G)$ denote the number of invariant factors of the distance matrix of $G$ equal to 1.
In this case, every graph with at least one non-loop edge has at least two invariant factors equal to one.

The following observation will give us the relation between the Smith normal form of the distance matrix and the distance ideals.
\begin{proposition}\cite{at}\label{prop:eval1}
Let ${\bf d}\in \mathbb{Z}^{V(G)}$.
If $f_1\mid\cdots\mid f_{r}$ are the invariant factors of the integer matrix $D(G,{\bf d})$, then
\[
I_i(G,{\bf d})=\left\langle \prod_{j=1}^{i} f_j \right\rangle\text{ for all }1\leq i\leq r.
\]
\end{proposition}
Thus to recover $\Delta_i(G)$ and $f_i$ from the distance ideals, we just need to evaluate $I_i(G,X_G)$ at $X_G={\bf 0}$.
Therefore, if the distance ideal $I_i(G,X_G)$ is trivial, then $\Delta_i(G)$ and $f_i$ are equal to $1$.
Equivalently, if $\Delta_i(G)$ and $f_i$ are not equal to $1$, then the distance ideal $I_i(D, X_D)$ is not trivial.

\begin{corollary}\cite{at}\label{coro:eval1}
    For any graph $G$, $\Phi_\mathcal{R}(G)\leq \phi_\mathcal{R}(G)$.
    And, for any positive integer $k$, $\Lambda_k$ contains the family of graphs with $\phi_\mathcal{R}(G)\leq k$.
\end{corollary}

Therefore, finding a characterization of $\Lambda_k$ can help in the characterization of the graphs with $\phi_\mathbb{Z}(G)\leq k$.

\subsection{Distance ideals and induced subgraphs}

In general, the distance ideals are not monotone under taking induced subgraphs.
However, we have the following results.

\begin{lemma}\cite{at}\label{lemma:inducemonotone}
Let $H$ be an induced subgraph of $G$ such that for every pair of vertices $v_i,v_j$ in $V(H)$, there is a shortest path from $v_i$ to $v_j$ in $G$ which lies entirely in $H$.
Then, $I^\mathcal{R}_i(H,X_H)\subseteq I^\mathcal{R}_i(G,X_G)$  for all $i\leq |V(H)|$, and $\Phi_\mathcal{R}(H)\leq \Phi_\mathcal{R}(G)$.
\end{lemma}

In particular we have the following.

\begin{lemma}\cite{at}\label{lemma:distance2inducedmonotone}
Let $H$ be an induced subgraph of $G$ with diameter is 2, that is the distance between any pair of vertices in $H$ is at most 2.
Then $I^\mathcal{R}_i(H,X_H)\subseteq I^\mathcal{R}_i(G,X_G)$  for all $i\leq |V(H)|$.
\end{lemma}

\begin{figure}[h!]
\begin{center}
\begin{tabular}{c@{\extracolsep{10mm}}c@{\extracolsep{10mm}}c@{\extracolsep{10mm}}c@{\extracolsep{10mm}}c}
	\begin{tikzpicture}[scale=.5]
	\tikzstyle{every node}=[minimum width=0pt, inner sep=2pt, circle]
	\draw (126+36:1) node (v1) [draw] {};
	\draw (198+36:1) node (v2) [draw] {};
	\draw (270+36:1) node (v3) [draw] {};
	\draw (342+36:1) node (v4) [draw] {};
	\draw (v1) -- (v2);
	\draw (v2) -- (v3);
	\draw (v4) -- (v3);
	\end{tikzpicture}
&
	\begin{tikzpicture}[scale=.5]
	\tikzstyle{every node}=[minimum width=0pt, inner sep=2pt, circle]
	\draw (-.5,-.9) node (v1) [draw] {};
	\draw (.5,-.9) node (v2) [draw] {};
	\draw (0,0) node (v3) [draw] {};
	\draw (0,.9) node (v4) [draw] {};
	\draw (v1) -- (v2);
	\draw (v1) -- (v3);
	\draw (v2) -- (v3);
	\draw (v3) -- (v4);
	\end{tikzpicture}
&
	\begin{tikzpicture}[scale=.5]
	\tikzstyle{every node}=[minimum width=0pt, inner sep=2pt, circle]
	\draw (-.5,0) node (v2) [draw] {};
	\draw (0,-.9) node (v1) [draw] {};
	\draw (.5,0) node (v3) [draw] {};
	\draw (0,.9) node (v4) [draw] {};
	\draw (v1) -- (v2);
	\draw (v1) -- (v3);
	\draw (v2) -- (v3);
	\draw (v2) -- (v4);
	\draw (v3) -- (v4);
	\end{tikzpicture}
&
    \begin{tikzpicture}[scale=.5]
	\tikzstyle{every node}=[minimum width=0pt, inner sep=2pt, circle]
	\draw (0:1) node (v1) [draw] {};
	\draw (90:1) node (v2) [draw] {};
	\draw (180:1) node (v3) [draw] {};
	\draw (270:1) node (v4) [draw] {};
	\draw (v1) -- (v2) -- (v3) -- (v4) -- (v1);
	\end{tikzpicture}
\\
$P_4$
&
$\sf{paw}$
&
$\sf{diamond}$
&
$C_4$
\end{tabular}
\end{center}
\caption{The graphs $P_4$, $\sf{paw}$, $\sf{diamond}$ and $C_4$.}
\label{fig:forbiddendistance1}
\end{figure}
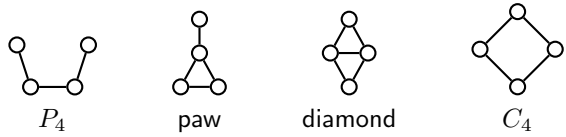

In fact, in \cite{at} the family $\Lambda_1$ of graphs having only 1 trivial distance ideal was characterized in terms of induced forbidden subgraphs: $\{P_4,\sf{paw},\sf{diamond}\}$-free graphs; that are the graphs isomorphic to an induced subgraph of $K_{m,n}$ or $K_{n}$.
Also in \cite{at}, the family of graphs having only 1 trivial distance ideal over $\mathbb{R}$ was characterized as: $\{P_4,{\sf{paw}},{\sf{diamond}}, C_4\}$-free graphs, that are the graphs isomorphic to an induced subgraph of $K_{1,n}$ or $K_{n}$.

These families appear in other contexts.
A graph is {\it trivially perfect} if for every induced subgraph the stability number equals the number of maximal cliques.
In \cite[Theorem 2]{G}, Golumbic characterized trivially prefect graphs as $\{P_4,C_4\}$-free graphs. There are other equivalent characterization of this family, see \cite{CCY,Rubio}.
Therefore, graphs with 1 trivial distance ideal over $\mathbb{R}$ are a subclass of trivial perfect graphs.

\begin{figure}[h]
\begin{center}
\begin{tabular}{c@{\extracolsep{.5cm}}c@{\extracolsep{.5cm}}c@{\extracolsep{.5cm}}c@{\extracolsep{.5cm}}c}
\begin{tikzpicture}[scale=.5]
	\tikzstyle{every node}=[minimum width=0pt, inner sep=2pt, circle]
    \draw (54+36:1) node (v0) [draw] {};
	\draw (126+36:1) node (v1) [draw] {};
	\draw (198+36:1) node (v3) [draw] {};
	\draw (270+36:1) node (v2) [draw] {};
	\draw (342+36:1) node (v4) [draw] {};
    \Edge[](v0)(v4);
    \Edge[](v1)(v3);
    \Edge[](v2)(v3);
    \Edge[](v2)(v4);
    \Edge[](v3)(v4);
\end{tikzpicture}
&
\begin{tikzpicture}[scale=.5]
	\tikzstyle{every node}=[minimum width=0pt, inner sep=2pt, circle]
    \draw (54+36:1) node (v0) [draw] {};
	\draw (126+36:1) node (v1) [draw] {};
	\draw (198+36:1) node (v3) [draw] {};
	\draw (270+36:1) node (v2) [draw] {};
	\draw (342+36:1) node (v4) [draw] {};
    \Edge[](v0)(v4);
    \Edge[](v1)(v3);
    \Edge[](v1)(v4);
    \Edge[](v2)(v3);
    \Edge[](v2)(v4);
    \Edge[](v3)(v4);
\end{tikzpicture}
&
\begin{tikzpicture}[scale=.5]
	\tikzstyle{every node}=[minimum width=0pt, inner sep=2pt, circle]
    \draw (54+36:1) node (v0) [draw] {};
	\draw (126+36:1) node (v1) [draw] {};
	\draw (198+36:1) node (v3) [draw] {};
	\draw (270+36:1) node (v2) [draw] {};
	\draw (342+36:1) node (v4) [draw] {};
    \Edge[](v0)(v1);
    \Edge[](v0)(v4);
    \Edge[](v1)(v3);
    \Edge[](v2)(v3);
    \Edge[](v2)(v4);
    \Edge[](v3)(v4);
\end{tikzpicture}
&
\begin{tikzpicture}[scale=.5]
	\tikzstyle{every node}=[minimum width=0pt, inner sep=2pt, circle]
    \draw (54+36:1) node (v4) [draw] {};
	\draw (126+36:1) node (v1) [draw] {};
	\draw (198+36:1) node (v2) [draw] {};
	\draw (270+36:1) node (v3) [draw] {};
	\draw (342+36:1) node (v0) [draw] {};
    \Edge[](v0)(v3);
    \Edge[](v0)(v4);
    \Edge[](v1)(v2);
    \Edge[](v1)(v4);
    \Edge[](v2)(v3);
    \Edge[](v2)(v4);
    \Edge[](v3)(v4);
\end{tikzpicture}
&
\begin{tikzpicture}[scale=.5]
	\tikzstyle{every node}=[minimum width=0pt, inner sep=2pt, circle]
    \draw (54+36:1) node (v4) [draw] {};
	\draw (126+36:1) node (v1) [draw] {};
	\draw (198+36:1) node (v2) [draw] {};
	\draw (270+36:1) node (v3) [draw] {};
	\draw (342+36:1) node (v0) [draw] {};
    \Edge[](v0)(v3);
    \Edge[](v0)(v4);
    \Edge[](v1)(v2);
    \Edge[](v1)(v3);
    \Edge[](v1)(v4);
    \Edge[](v2)(v3);
    \Edge[](v2)(v4);
    \Edge[](v3)(v4);
\end{tikzpicture}\\
\Ga & \Gb & \Gc & \Gd & \Ge \\

\begin{tikzpicture}[scale=.5]
	\tikzstyle{every node}=[minimum width=0pt, inner sep=2pt, circle]
    \draw (60:1) node (v0) [draw] {};
	\draw (120:1) node (v1) [draw] {};
	\draw (180:1) node (v2) [draw] {};
	\draw (240:1) node (v3) [draw] {};
	\draw (300:1) node (v5) [draw] {};
    \draw (0:1) node (v4) [draw] {};
    \Edge[](v0)(v4);
    \Edge[](v1)(v5);
    \Edge[](v2)(v3);
    \Edge[](v2)(v5);
    \Edge[](v3)(v5);
    \Edge[](v4)(v5);
\end{tikzpicture}
&
\begin{tikzpicture}[scale=.5]
	\tikzstyle{every node}=[minimum width=0pt, inner sep=2pt, circle]
    \draw (60:1) node (v4) [draw] {};
	\draw (120:1) node (v1) [draw] {};
	\draw (180:1) node (v2) [draw] {};
	\draw (240:1) node (v3) [draw] {};
	\draw (300:1) node (v0) [draw] {};
    \draw (0:1) node (v5) [draw] {};
    \Edge[](v0)(v5);
    \Edge[](v1)(v2);
    \Edge[](v1)(v4);
    \Edge[](v2)(v3);
    \Edge[](v3)(v5);
    \Edge[](v4)(v5);
\end{tikzpicture}
&
\begin{tikzpicture}[scale=.5]
	\tikzstyle{every node}=[minimum width=0pt, inner sep=2pt, circle]
    \draw (60:1) node (v5) [draw] {};
	\draw (120:1) node (v1) [draw] {};
	\draw (180:1) node (v2) [draw] {};
	\draw (240:1) node (v3) [draw] {};
	\draw (300:1) node (v4) [draw] {};
    \draw (0:1) node (v0) [draw] {};
    \Edge[](v0)(v4);
    \Edge[](v1)(v2);
    \Edge[](v1)(v5);
    \Edge[](v2)(v5);
    \Edge[](v3)(v4);
    \Edge[](v3)(v5);
\end{tikzpicture}
&
\begin{tikzpicture}[scale=.5]
	\tikzstyle{every node}=[minimum width=0pt, inner sep=2pt, circle]
    \draw (60:1) node (v0) [draw] {};
	\draw (120:1) node (v5) [draw] {};
	\draw (180:1) node (v2) [draw] {};
	\draw (240:1) node (v3) [draw] {};
	\draw (300:1) node (v4) [draw] {};
    \draw (0:1) node (v1) [draw] {};
    \Edge[](v0)(v5);
    \Edge[](v1)(v4);
    \Edge[](v1)(v5);
    \Edge[](v2)(v3);
    \Edge[](v2)(v5);
    \Edge[](v3)(v5);
    \Edge[](v4)(v5);
\end{tikzpicture}
&
\begin{tikzpicture}[scale=.5]
	\tikzstyle{every node}=[minimum width=0pt, inner sep=2pt, circle]
    \draw (60:1) node (v0) [draw] {};
	\draw (120:1) node (v4) [draw] {};
	\draw (180:1) node (v2) [draw] {};
	\draw (240:1) node (v3) [draw] {};
	\draw (300:1) node (v1) [draw] {};
    \draw (0:1) node (v5) [draw] {};
    \Edge[](v0)(v5);
    \Edge[](v1)(v5);
    \Edge[](v2)(v3);
    \Edge[](v2)(v4);
    \Edge[](v3)(v4);
    \Edge[](v3)(v5);
    \Edge[](v4)(v5);
\end{tikzpicture}\\
$\Gf$ & \Gg & $\Gh$ & $\Gi$ & $\Gj$ \\

\begin{tikzpicture}[scale=.5]
	\tikzstyle{every node}=[minimum width=0pt, inner sep=2pt, circle]
    \draw (60:1) node (v0) [draw] {};
	\draw (120:1) node (v1) [draw] {};
	\draw (180:1) node (v4) [draw] {};
	\draw (240:1) node (v3) [draw] {};
	\draw (300:1) node (v2) [draw] {};
    \draw (0:1) node (v5) [draw] {};
    \Edge[](v0)(v1);
    \Edge[](v0)(v5);
    \Edge[](v1)(v4);
    \Edge[](v2)(v4);
    \Edge[](v2)(v5);
    \Edge[](v3)(v4);
    \Edge[](v3)(v5);
\end{tikzpicture}
&
\begin{tikzpicture}[scale=.5]
	\tikzstyle{every node}=[minimum width=0pt, inner sep=2pt, circle]
    \draw (60:1) node (v0) [draw] {};
	\draw (120:1) node (v1) [draw] {};
	\draw (180:1) node (v4) [draw] {};
	\draw (240:1) node (v3) [draw] {};
	\draw (300:1) node (v2) [draw] {};
    \draw (0:1) node (v5) [draw] {};
    \Edge[](v0)(v5);
    \Edge[](v1)(v4);
    \Edge[](v2)(v3);
    \Edge[](v2)(v4);
    \Edge[](v2)(v5);
    \Edge[](v3)(v4);
    \Edge[](v3)(v5);
\end{tikzpicture}
&
\begin{tikzpicture}[scale=.5]
	\tikzstyle{every node}=[minimum width=0pt, inner sep=2pt, circle]
    \draw (60:1) node (v0) [draw] {};
	\draw (120:1) node (v1) [draw] {};
	\draw (180:1) node (v4) [draw] {};
	\draw (240:1) node (v3) [draw] {};
	\draw (300:1) node (v2) [draw] {};
    \draw (0:1) node (v5) [draw] {};
    \Edge[](v0)(v5);
    \Edge[](v1)(v4);
    \Edge[](v1)(v5);
    \Edge[](v2)(v3);
    \Edge[](v2)(v4);
    \Edge[](v2)(v5);
    \Edge[](v3)(v4);
    \Edge[](v3)(v5);
\end{tikzpicture}
&
\begin{tikzpicture}[scale=.5]
	\tikzstyle{every node}=[minimum width=0pt, inner sep=2pt, circle]
    \draw (60:1) node (v0) [draw] {};
	\draw (120:1) node (v1) [draw] {};
	\draw (180:1) node (v4) [draw] {};
	\draw (240:1) node (v3) [draw] {};
	\draw (300:1) node (v2) [draw] {};
    \draw (0:1) node (v5) [draw] {};
    \Edge[](v0)(v1);
    \Edge[](v0)(v5);
    \Edge[](v1)(v4);
    \Edge[](v2)(v3);
    \Edge[](v2)(v4);
    \Edge[](v2)(v5);
    \Edge[](v3)(v4);
    \Edge[](v3)(v5);
\end{tikzpicture}
&
\begin{tikzpicture}[scale=.5]
	\tikzstyle{every node}=[minimum width=0pt, inner sep=2pt, circle]
    \draw (60:1) node (v2) [draw] {};
	\draw (120:1) node (v1) [draw] {};
	\draw (180:1) node (v4) [draw] {};
	\draw (240:1) node (v3) [draw] {};
	\draw (300:1) node (v0) [draw] {};
    \draw (0:1) node (v5) [draw] {};
    \Edge[](v0)(v2);
    \Edge[](v0)(v3);
    \Edge[](v0)(v4);
    \Edge[](v0)(v5);
    \Edge[](v1)(v2);
    \Edge[](v1)(v3);
    \Edge[](v1)(v4);
    \Edge[](v1)(v5);
    \Edge[](v2)(v4);
    \Edge[](v2)(v5);
    \Edge[](v3)(v4);
    \Edge[](v3)(v5);
    \Edge[](v4)(v5);
\end{tikzpicture}\\
$\Gk$ & \Gl & $\Gm$ & $\Gn$ & $\Go$ \\
\end{tabular}
\end{center}
\begin{center}
\begin{tabular}{c@{\extracolsep{.5cm}}c}

\begin{tikzpicture}[scale=.5]
	\tikzstyle{every node}=[minimum width=0pt, inner sep=2pt, circle]
    \draw (0:1) node (v0) [draw] {};
	\draw (360/7:1) node (v1) [draw] {};
	\draw (2*360/7:1) node (v4) [draw] {};
	\draw (3*360/7:1) node (v3) [draw] {};
	\draw (4*360/7:1) node (v2) [draw] {};
    \draw (5*360/7:1) node (v5) [draw] {};
    \draw (6*360/7:1) node (v6) [draw] {};
    \Edge[](v0)(v1);
    \Edge[](v0)(v6);
    \Edge[](v1)(v6);
    \Edge[](v2)(v4);
    \Edge[](v2)(v5);
    \Edge[](v3)(v4);
    \Edge[](v3)(v5);
    \Edge[](v4)(v6);
    \Edge[](v5)(v6);
\end{tikzpicture}\\
$\Gq$\\
\end{tabular}
\end{center}
\label{fig:for2}
\caption{Some minimal forbidden graphs for graphs with 2 trivial distance ideals over $\mathbb{Z}$.}
\end{figure}
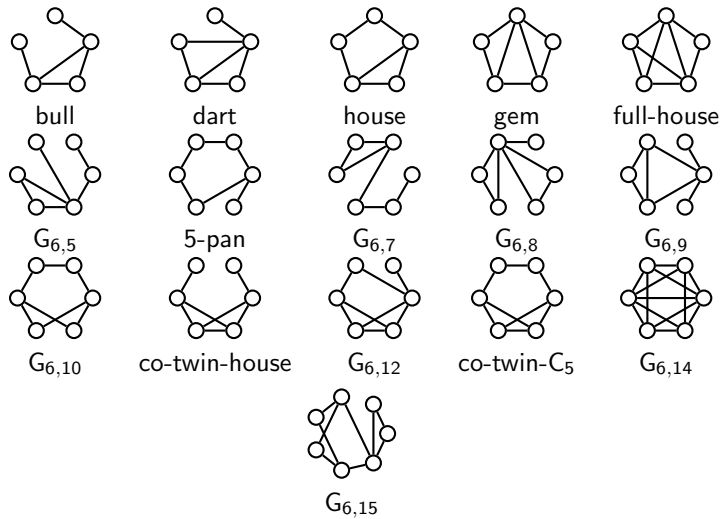

The aim of this paper is to explore the properties of the family $\Lambda_2$ of graphs with at most two trivial distance ideals over $\mathbb{Z}$.
In particular, we are going to find an infinite number of graphs that are forbidden for $\Lambda_2$.
Let $F$ be the set of 17 graphs shown in Figure~\ref{fig:for2}.
In Section~\ref{section:characterization}, we will prove that graphs in $\Lambda_2$ are $\{F,\text{\sf odd-holes}\}$-free graphs, where {\sf odd-holes} are cycles of odd length greater or equal than 7.

One of the main applications in finding a characterization of $\Lambda_2$ is that it is an approach to obtain a characterization of the graphs with $\phi_\mathbb{Z}(G)=2$ since they are contained in $\Lambda_2$.
In \cite[Theorem 3]{HW}, it was proved that the distance matrix of trees has exactly 2 invariant factors equal to 1.
Therefore,
\[
\text{trees} \subseteq \Lambda_2 \subseteq \{F,\text{\sf odd-holes}\} \text{-free graphs}.
\]

Among the forbidden graphs for $\Lambda_2$ there are several graphs studied in other contexts, like $\Ga$ and {\sf odd-holes} studied in \cite{CI,CII} and \cite{CS}, respectively.
Other related family is the 3-leaf powers that was characterized in \cite{DGHN} as $\{\Ga,\Gb,\Gd\}$-free chordal graphs.

Distance-hereditary graphs is another related family, defined by Howorka in \cite{H77}.
A graph $G$ is {\it distance-hereditary} if for each connected induced subgraph $H$ of $G$ and every pair $u$ and $v$ of vertices in $H$, $d_H(u,v)=d_G(u,v)$.
Distance-hereditary graphs are $\{\Gc,\Gd,{\sf domino}, {\sf holes}\}$-free graphs, where {\sf holes} are cycles of length greater or equal than 5, which intersects with $\Lambda_2$.
Also, if $H$ is a connected induced subgraph of a distance-hereditary graph $G$, then $I^\mathcal{R}_i(H,X_H)\subseteq I^\mathcal{R}_i(G,X_G)$  for all $i\leq |V(H)|$.

Previously, an analogous notion to the distance ideals but for the adjacency and Laplacian matrices was explored.
These were called {\it critical ideals}, see \cite{CV}.
They have been explored in \cite{alfacorrval,AV,AV1,AVV}, and in \cite{alfaro,alfalin} it was found new connections in contexts different from the Smith group or Sandpile group like the zero-forcing number and the minimum rank of a graph.
In this setting, the set of forbidden graphs for the family with at most $k$ trivial critical ideals is conjectured to be finite, see \cite[Conjecture 5.5]{AV}.
It is interesting that for distance ideals this is not true.

\section{Graphs with 2 trivial distance ideals}\label{section:characterization}

A graph $G$ is {\it forbidden} for $\Lambda_k$ if the $(k+1)$-th distance ideal of $G$ is trivial.
In addition, a forbidden graph $H$ for $\Lambda_k$ is {\it minimal} if $H$ does not contain a connected forbidden graph for $\Lambda_k$ as induced subgraph, and any graph $G$ containing $H$ as induced subgraph, have that $G$ id forbidden for $\Lambda_k$.
The set of minimal forbidden graphs for $\Lambda_k$ will be denoted by ${\sf Forb}_k$.

\begin{lemma}\label{lemma:diameter2trivialdistance}
    Let $G\in F$.
    Then, $\Phi(G)=3$ and no induced subgraph $H$ of $G$ has $\Phi(H)=3$.
\end{lemma}
\begin{proof}
    This could be checked using Macaulay2 code \cite{macaulay2} given in \cite[Appendix]{at}.
\end{proof}

\begin{proposition}\label{proposition:forbiddensimplecases}
    The graphs \Gb, \Gc, \Gd, \Ge, $\Gi$, $\Gk$, $\Gn$, $\Go$  are in ${\sf Forb}_2$.
\end{proposition}
\begin{proof}
    The result follows by Lemmas~\ref{lemma:distance2inducedmonotone} and \ref{lemma:diameter2trivialdistance}.
\end{proof}

\begin{lemma}\label{lemma:bullisforbidden}
    $\Ga$ graph is in ${\sf Forb}_2$.
\end{lemma}
\begin{proof}
    If there is no vertex $u$ in $G$ adjacent to the leaf vertices of \Ga, then distance  in $G$ between the two leaf vertices is 3.
    Therefore, by Lemma~\ref{lemma:inducemonotone}, $\Phi(G)\geq\Phi(\Ga)=3$, and $G\in{\sf Forb}_2$.
    On the other hand, if the distance, in $G$, between the two leaf vertices of $\Ga$ is 2, then the submatrix of $D(G,X_G)$ associated with the vertices of $\Ga$ is
\[
M=D(G,X_G)[V(\Ga)]=
\begin{bmatrix}
    u & 2 & 2 & 2 & 1\\
    2 & v & 2 & 1 & 2\\
    2 & 2 & x_1 & 1 & 1\\
    2 & 1 & 1 & x_2 & 1\\
    1 & 2 & 1 & 1 & x_3

\end{bmatrix},
\]
which have $\langle \minors_3(M)\rangle=\langle 1\rangle$.
From which follows that $\langle 1\rangle=\langle \minors_3(M)\rangle\subseteq I_3(G,X_G)$, and thus $G\in{\sf Forb}_2$.
\end{proof}

\begin{lemma}
    $\Gf$ is in ${\sf Forb}_2$.
\end{lemma}
\begin{proof}
Let $G$ be a graph containing $\Gf$ as induced subgraph, and let $M$ be the submatrix of $D(G,X_G)$ associated with the vertices of $\Gf$ that is described in Figure \ref{figue:gf} and where $d_G(v_0,v_3)=y_0$, $d_G(v_0,v_2)=y_1$ and $d_G(v_0,v_1)=y_2$.

\begin{figure}[h!]
\begin{minipage}[t]{0.4\textwidth}
$
\begin{bmatrix}
    x_0 & y_2 & y_1 & y_0 & 2 & 1\\
    y_2 & x_1 & 2 & 2 & 1 & 2\\
    y_1 & 2 & x_2 & 1 & 1 & 2\\
    y_0 & 2 & 1 & x_3 & 1 & 2\\
     2 & 1 & 1 & 1 & x_4 & 1\\
     1 & 2 & 2 & 2 & 1 & x_5

\end{bmatrix}
$
\end{minipage}
\begin{minipage}[c]{0.4\textwidth}
\begin{tikzpicture}[scale=1]
	\tikzstyle{every node}=[minimum width=0pt, inner sep=2pt, circle]
    \draw (60:1) node (v0) [draw] {\small 0};
	\draw (120:1) node (v1) [draw] {\small 1};
	\draw (180:1) node (v2) [draw] {\small 2};
	\draw (240:1) node (v3) [draw] {\small 3};
	\draw (300:1) node (v5) [draw] {\small 4};
    \draw (0:1) node (v4) [draw] {\small 5};
    \Edge[](v0)(v4);
    \Edge[](v1)(v5);
    \Edge[](v2)(v3);
    \Edge[](v2)(v5);
    \Edge[](v3)(v5);
    \Edge[](v4)(v5);
\end{tikzpicture}
\end{minipage}
\caption{Submatrix associated with $\Gf$.}
\label{figue:gf}
\end{figure}

We have that $\det(M[\{1,2,5\},\{0,3,4\}])=-y_2+1$ and $\det(M[\{1,2,4\},\{0,3,5\}])=-y_2+4$.
Thus
$\langle -y_2+1, -y_2+4\rangle\subseteq\langle \minors_3(M) \rangle\subseteq I_3(G,X_G)$.
There are two cases either $y_2$ is equal to $2$ or $3$.
In both cases we have $1\in I_3(G,X_G)$, from which follows that $G\in{\sf Forb}_2$.
\end{proof}

\begin{lemma}\label{Lemma:5-panforbidden}
    {\Gg} is in ${\sf Forb}_2$.
\end{lemma}
\begin{proof}
Let $G$ be a graph having {\Gg} as induced subgraph, and let $M$ be the submatrix of $D(G,X_G)$ associated with the vertices of {\Gg} that is shown in Figure \ref{figue:gg} and where $d_G(v_0,v_1)=y_1$ and $d_G(v_0,v_2)=y_0$.

\begin{figure}[h!]
\begin{minipage}[c]{0.4\textwidth}
$
\begin{bmatrix}
    x_0 & y_1 & y_0 & 2 & 2 & 1\\
    y_1 & x_1 & 1 & 2 & 1 & 2\\
    y_0 & 1 & x_2 & 1 & 2 & 2\\
    2 & 2 & 1 & x_3 & 2 & 1\\
    2 & 1 & 2 & 2 & x_4 & 1\\
    1 & 2 & 2 & 1 & 1 & x_5
\end{bmatrix},
$
\end{minipage}
\begin{minipage}[c]{0.4\textwidth}
\begin{tikzpicture}[scale=1]
	\tikzstyle{every node}=[minimum width=0pt, inner sep=2pt, circle]
    \draw (60:1) node (v4) [draw] {\small 4};
	\draw (120:1) node (v1) [draw] {\small 1};
	\draw (180:1) node (v2) [draw] {\small 2};
	\draw (240:1) node (v3) [draw] {\small 3};
	\draw (300:1) node (v0) [draw] {\small 0};
    \draw (0:1) node (v5) [draw] {\small 5};
    \Edge[](v0)(v5);
    \Edge[](v1)(v2);
    \Edge[](v1)(v4);
    \Edge[](v2)(v3);
    \Edge[](v3)(v5);
    \Edge[](v4)(v5);
\end{tikzpicture}
\end{minipage}
\caption{Submatrix associated with $\Gg$.}
\label{figue:gg}
\end{figure}


We have that $\det(M[\{2,3,4\},\{1,2,5\}])=5-x_2$, $\det(M[\{2,4,5\},\{1,2,3\}])=3x_2-4$ and $\det(M[\{0,1,2\},\{3,4,5\}])=-5$.
Since $\langle 1\rangle=\langle 5-x_2, 3x_2-4,-5\rangle\subseteq\langle \minors_3(M) \rangle\subseteq I_3(G,X_G)$, then $G\in{\sf Forb}_2$.
\end{proof}

\begin{lemma}
    $\Gh$ is in ${\sf Forb}_2$.
\end{lemma}
\begin{proof}
Let $G$ be a graph having $\Gh$ as induced subgraph, and let $M$ be the submatrix of $D(G,X_G)$ associated with the vertices of $\Gh$ that is shown in Figure~\ref{figue:gh} and where $d_G(v_0,v_1)=y_4$, $d_G(v_0,v_2)=y_3$, $d_G(v_0,v_5)=y_2$, $d_G(v_1,v_4)=y_1$ and $d_G(v_2,v_4)=y_0$.

\begin{figure}[h!]
\begin{minipage}[c]{0.4\textwidth}
$
\begin{bmatrix}
    x_0 & y_4 & y_3 & 2 & 1 & y_2\\
    y_4 & x_1 & 1 & 2 & y_1 & 1\\
    y_3 & 1 & x_2 & 2 & y_0 & 1\\
    2 & 2 & 2 & x_3 & 1 & 1\\
    1 & y_1 & y_0 & 1 & x_4 & 2\\
    y_2 & 1 & 1 & 1 & 2 & x_5
\end{bmatrix},
$
\end{minipage}
\begin{minipage}[c]{0.4\textwidth}
\begin{tikzpicture}[scale=1]
	\tikzstyle{every node}=[minimum width=0pt, inner sep=2pt, circle]
    \draw (60:1) node (v5) [draw] {\small 5};
	\draw (120:1) node (v1) [draw] {\small 1};
	\draw (180:1) node (v2) [draw] {\small 2};
	\draw (240:1) node (v3) [draw] {\small 3};
	\draw (300:1) node (v4) [draw] {\small 4};
    \draw (0:1) node (v0) [draw] {\small 0};
    \Edge[](v0)(v4);
    \Edge[](v1)(v2);
    \Edge[](v1)(v5);
    \Edge[](v2)(v5);
    \Edge[](v3)(v4);
    \Edge[](v3)(v5);
\end{tikzpicture}
\end{minipage}
\caption{Submatrix associated with $\Gh$.}
\label{figue:gh}
\end{figure}

Let
\begin{eqnarray*}
I=\{
y_0 y_2 + 2 y_0 - y_1 y_2 - 2 y_1,
y_0 y_3 - 2 y_0 - 4 y_3 + 5,\\
y_0 y_4 - 2 y_0 - 4 y_4 + 5,
3 y_0 - 3 y_1,\\
2 y_1 y_2 - 2 y_1 - y_2 - 3 y_4 + 4,
y_1 y_3 - 2 y_1 - 4 y_3 + 5,\\
y_1 y_4 - 2 y_1 - 4 y_4 + 5,
y_2 y_3 + 6 y_2 y_4 - 8 y_2 + 2 y_3 - 3 y_4^2 + 2,\\
7 y_2 y_4 - 8 y_2 - 3 y_4^2 + 2 y_4 + 2,
3 y_3 - 3 y_4\},\\
\end{eqnarray*}
and
\begin{eqnarray*}
J=\{
x_0 - 2y_2y_4 + 2y_2 + y_3y_4 - 2y_3,
x_1x_2 - 5x_1 - 5x_2 + 9,\\
x_1x_3 - 2x_1 - x_3 + 2,
x_1y_0 - x_1 - 2y_0 + y_1 + 1,\\
x_1y_2 - x_1 - y_2 + 1,
x_1y_3 - 2x_1 - 8y_3 + 7y_4 + 2,\\
3x_1 - 3,
x_2x_3 - 2x_2 - x_3 + 2,
x_2y_1 - x_2 + y_0 - 2y_1 + 1,\\
x_2y_2 - x_2 - y_2 + 1,
x_2y_4 - 2x_2 + 7y_3 - 8y_4 + 2,
3x_2 - 3,\\
x_3y_0 - 2x_3 - 2y_0 + 7,
x_3y_1 - 2x_3 - 2y_1 + 7,\\
x_3y_2 - x_3y_4 - 4y_2 + 2y_4 + 2,
x_3y_3 - x_3y_4 - 2y_3 + 2y_4,\\
2x_3y_4 - x_3 - y_4 - 4,
x_4 + y_0y_1 - y_0 - 4y_1 + 2,\\
x_5 - 3y_1y_2 + 3y_1 - 2y_2 + 6y_4 - 6\}.\\
\end{eqnarray*}
It can be checked that the lexicographic Gr\"obner basis of the ideal $\langle \minors_3(M)\rangle$ is equal to $\langle I\cup J\rangle$.
Furthermore, let ${\bf d}$ be a vector such that the $i$-th entry is taken over a possible values of $y_i$, that is, $d_0\in \{2,3\}$, $d_1\in \{2,3\}$, $d_2\in \{2,3\}$, $d_3\in \{2,3,4\}$ and $d_4\in \{2,3,4\}$.
Note that if $\gcd(I|_{\bf y = d})$ is equal to 1, then $\langle \minors_3(M)\rangle=\langle 1 \rangle$, which implies that $I_3(G,X_G)$ is trivial.
It can be verified that $\gcd(I|_{\bf y=d})$ is equal to 1, except for the following vectors: $(2, 2, 2, 2, 2)$, $(2, 2, 3, 2, 2)$, $(2, 2, 3, 3, 3)$ and $(3, 3, 3, 3, 3)$.

Consider $p=x_3 y_0 - 2x_3 - 2y_0 + 7\in J$ and $q=x_3 y_2 - x_3 y_4 - 4 y_2 + 2 y_4 + 2\in J$.
Since $p|_{y_0=2}=3$ and $q|_{y_2=2,y_4=2}=-2$, then in the case when ${\bf d}=(2, 2, 2, 2, 2)$, $\gcd(J|_{\bf y = d})=1$, and $I_3(G,X_G)$ is trivial.
Since $q|_{y_2=3,y_4=3}=-4$, then in the case when ${\bf d}=(2, 2, 3, 3, 3)$, $\gcd(J|_{\bf y = d})=1$, and $I_3(G,X_G)$ is trivial.

Let us consider the case associated with vector $(2, 2, 3, 2, 2)$.
In particular, $d_1=2=d_G(v_1,v_4)$ implies there exists a vertex $u\in G$ adjacent with $v_1$ and $v_4$.
Let $M'$ be the submatrix of $D(G,X_G)$ associated with the vertices $V(\Gh)\cup\{u\}$, that is
\[
M'=
\begin{bmatrix}
    x_0 & 2 & 2 & 2 & 1 & 3 & a\\
    2 & 2 & 1 & 2 & 2 & 1 & 1\\
    2 & 1 & x_2 & 2 & 2 & 1 & c\\
    2 & 2 & 2 & x_3 & 1 & 1 & d\\
    1 & 2 & 2 & 1 & x_4 & 2 & 1\\
    3 & 1 & 1 & 1 & 2 & x_5 & f\\
    a & 1 & c & d & 1 & f & x_u
\end{bmatrix},
\]
Since $\langle \minors_3(M')\rangle=\langle 1 \rangle$, then $I_3(G,X_G)$ is trivial.

Finally, let us consider the case associated with vector $(3, 3, 3, 3, 3)$.
Since $d_G(v_0,v_1)=3$, then there exists a path $v_0, v, u, v_1$ such that $u\neq v_4$, since otherwise $d_G(v_4,v_1)$ would be equal to 2.
Let $M'$ be the submatrix of $D(G,X_G)$ associated with the vertices $V(\Gh)\cup\{u\}$, that is
\[
M'=
\begin{bmatrix}
    x_0 & 3 & 3 & 2 & 1 & 3 & 2\\
    3 & x_1 & 1 & 2 & 3 & 1 & 1\\
    3 & 1 & x_2 & 2 & 3 & 1 & c\\
    2 & 2 & 2 & x_3 & 1 & 1 & d\\
    1 & 3 & 3 & 1 & x_4 & 2 & e\\
    3 & 1 & 1 & 1 & 2 & x_5 & f\\
    2 & 1 & c & d & e & f & x_u
\end{bmatrix},
\]
Since $\langle \minors_3(M')\rangle=\langle 1 \rangle$, then $I_3(G,X_G)$ is trivial.

\end{proof}

\begin{lemma}
    $\Gj$ is in ${\sf Forb}_2$.
\end{lemma}
\begin{proof}
Let $G$ be a graph containing $\Gj$ as induced subgraph, and let $M$ be the submatrix of $D(G,X_G)$ associated with the vertices of $\Gj$ that is shown in Figure~\ref{figue:gj} and where $d_G(v_1,v_2)=y_0$ and $d_G(v_0,v_2)=y_1$.

\begin{figure}[h!]
\begin{minipage}[c]{0.4\textwidth}
$
\begin{bmatrix}
    x_0 & 2 & y_1 & 2 & 2 & 1\\
    2 & x_1 & y_0 & 2 & 2 & 1\\
    y_1 & y_0 & x_2 & 1 & 1 & 2\\
    2 & 2 & 1 & x_3 & 1 & 1\\
    2 & 2 & 1 & 1 & x_4 & 1\\
    1 & 1 & 2 & 1 & 1 & x_5
\end{bmatrix},
$
\end{minipage}
\begin{minipage}[c]{0.4\textwidth}
\begin{tikzpicture}[scale=1]
	\tikzstyle{every node}=[minimum width=0pt, inner sep=2pt, circle]
    \draw (60:1) node (v0) [draw] {\small 0};
	\draw (120:1) node (v4) [draw] {\small 4};
	\draw (180:1) node (v2) [draw] {\small 2};
	\draw (240:1) node (v3) [draw] {\small 3};
	\draw (300:1) node (v1) [draw] {\small 1};
    \draw (0:1) node (v5) [draw] {\small 5};
    \Edge[](v0)(v5);
    \Edge[](v1)(v5);
    \Edge[](v2)(v3);
    \Edge[](v2)(v4);
    \Edge[](v3)(v4);
    \Edge[](v3)(v5);
    \Edge[](v4)(v5);
\end{tikzpicture}
\end{minipage}
\caption{Submatrix associated with $\Gj$.}
\label{figue:gj}
\end{figure}

We have $\det(M[\{1,4,5\},\{1,3,4\}])=4-y_0$ and $\det(M[\{0,4,5\},\{1,2,3\}])=4-y_1$.
Since $y_0,y_1\in\{2,3\}$,
then when one of $y_0$ or $y_1$ is equal to 3, $\langle \minors_3(M) \rangle=\langle 1 \rangle$ and $I_3(G,X_G)$ is trivial.
On the other hand, $\det(M[\{1,4,5\},\{0,3,5\}])=1-2x_5$.
Thus, when one of $y_0$ or $y_1$ is equal to 2, then $\langle 2, 1-2x_5\rangle=\langle 1\rangle$.
In all cases we have $1\in I_3(G,X_G)$, from which follows that $G\in{\sf Forb}_2$.
\end{proof}

\begin{lemma}\label{lemma:forbidden_co-twin-house}
    {\Gl} is in ${\sf Forb}_2$.
\end{lemma}

\begin{proof}
Let $G$ be a graph having {\Gl} as induced subgraph, and let $M$ be the submatrix of $D(G,X_G)$ associated with the vertices of {\Gl} that is shown in Figure~\ref{figure:gl} and where $d_G(v_0,v_1)=y_2$, $d_G(v_0,v_4)=y_1$ and $d_G(v_1,v_5)=y_0$.

\begin{figure}[h!]
\begin{minipage}[c]{0.4\textwidth}
$
\begin{bmatrix}
    x_0 & y_2 & 2 & 2 & y_1 & 1\\
    y_2 & x_1 & 2 & 2 & 1 & y_0\\
    2 & 2 & x_2 & 1 & 1 & 1\\
    2 & 2 & 1 & x_3 & 1 & 1\\
    y_1 & 1 & 1 & 1 & x_4 & 2\\
    1 & y_0 & 1 & 1 & 2 & x_5
\end{bmatrix},
$
\end{minipage}
\begin{minipage}[c]{0.4\textwidth}
\begin{tikzpicture}[scale=1]
	\tikzstyle{every node}=[minimum width=0pt, inner sep=2pt, circle]
    \draw (60:1) node (v0) [draw] {\small 0};
	\draw (120:1) node (v1) [draw] {\small 1};
	\draw (180:1) node (v4) [draw] {\small 4};
	\draw (240:1) node (v3) [draw] {\small 3};
	\draw (300:1) node (v2) [draw] {\small 2};
    \draw (0:1) node (v5) [draw] {\small 5};
    \Edge[](v0)(v5);
    \Edge[](v1)(v4);
    \Edge[](v2)(v3);
    \Edge[](v2)(v4);
    \Edge[](v2)(v5);
    \Edge[](v3)(v4);
    \Edge[](v3)(v5);
\end{tikzpicture}
\end{minipage}
\caption{Submatrix associated with $\Gl$.}
\label{figure:gl}
\end{figure}

Let
\begin{eqnarray*}
I=\{
y_0y_1 - 2y_0 - 2y_1 + 3,
y_2 - 5
\},\\
\end{eqnarray*}
and
\begin{eqnarray*}
J=\{
x_0 + y_1 - 6,
x_1 + y_0 - 6,
x_2x_3 - 2x_2 - 2x_3 + 3,\\
x_2y_0 - x_2 - y_0 + 1,
x_2y_1 - x_2 - y_1 + 1,
3x_2 - 3,\\
x_3y_0 - x_3 - y_0 + 1,
x_3y_1 - x_3 - y_1 + 1,
3x_3 - 3,\\
x_4 + y_1 - 3,
x_5 + y_0 - 3
\}.\\
\end{eqnarray*}

It can be verified that the lexicographic Gr\"obner basis of the ideal $\langle \minors_3(M)\rangle$ is equal to $\langle I\cup J\rangle$.
Let ${\bf d}$ be a vector such that the $i$-th entry is taken over a possible values of $y_i$.
Thus $d_0,d_1\in \{2,3\}$ and $d_2\in \{2,3,4\}$.
Note that if $\gcd(I|_{\bf y = d})$ is equal to 1, then $\langle \minors_3(M)\rangle=\langle 1 \rangle$, which implies that $I_3(G,X_G)$ is trivial.
It can be verified that $\gcd(I|_{\bf y=d})$ is equal to 1 for all valid vectors, except for the following vectors:  $(3, 3, 2)$ and $(3, 3, 3)$.

Let us consider the case associated with vector $(3, 3, 3)$.
Since $d_G(v_0,v_1)=3$, then there exists a path $v_0, v, u, v_1$ such that $v\neq v_5$, since otherwise $d_G(v_5,v_1)$ would be equal to 2.
Let $M'$ be the submatrix of $D(G,X_G)$ associated with the vertices $V(\text{\Gl})\cup\{v\}$, that is
\[
M'=
\begin{bmatrix}
    x_0 & 3 & 2 & 2 & 3 & 1 & 1\\
    3 & x_1 & 2 & 2 & 1 & 3 & 2\\
    2 & 2 & x_2 & 1 & 1 & 1 & c\\
    2 & 2 & 1 & x_3 & 1 & 1 & d\\
    3 & 1 & 1 & 1 & x_4 & 2 & e\\
    1 & 3 & 1 & 1 & 2 & x_5 & f\\
    1 & 2 & c & d & e & f & x_v
\end{bmatrix},
\]
Since $\langle \minors_3(M')\rangle=\langle 1 \rangle$, then $I_3(G,X_G)$ is trivial.

Finally, let us consider the case associated with vector $(3, 3, 2)$.
Since $d_G(v_0,v_1)=2$, then there exists a vertex $u\in G$ adjacent with $v_0$ and $v_1$.
Let $M'$ be the submatrix of $D(G,X_G)$ associated with the vertices $V(\text{\Gl})\cup\{u\}$, that is
\[
M'=
\begin{bmatrix}
    x_0 & 2 & 2 & 2 & 3 & 1 & 1\\
    2 & x_1 & 2 & 2 & 1 & 3 & 1\\
    2 & 2 & x_2 & 1 & 1 & 1 & c\\
    2 & 2 & 1 & x_3 & 1 & 1 & d\\
    3 & 1 & 1 & 1 & x_4 & 2 & e\\
    1 & 3 & 1 & 1 & 2 & x_5 & f\\
    1 & 1 & c & d & e & f & x_u
\end{bmatrix},
\]
It can be seen that $\langle \minors_3(M')\rangle$ is equal to
\[
\langle
x_0, x_1, x_2 + 2, x_3 + 2, x_4, x_5, c + 1, d + 1, e + 1, f + 1, x_u + 2, 3
\rangle.
\]
From which follows that if $u$ is adjacent with $x_2$, $x_3$, $x_4$ or $x_5$, then one of $c + 1$, $d + 1$, $e + 1$ or $f + 1$ is equal to 2, and 1 would be in $I_3(G,X_G)$.
So, suppose $u$ is not adjacent with neither $x_2$, $x_3$, $x_4$ nor $x_5$.
Thus $e=f=2$.
However, $c$ and $d$ could be $2$ or $3$.
Note that if one of $c$ or $d$ is equal to 3, then 1 would be in $I_3(G,X_G)$.
Thus, assume $c$ and $d$ are equal to 2.
Therefore, there exists a vertex $v$ adjacent with $u$ and $x_2$.
Let $M''$ be the submatrix of $D(G,X_G)$ associated with the vertices $V(\text{\Gl})\cup\{u,v\}$, that is
\[
M''=
\begin{bmatrix}
    x_0 & 2 & 2 & 2 & 3 & 1 & 1 & a\\
    2 & x_1 & 2 & 2 & 1 & 3 & 1 & b\\
    2 & 2 & x_2 & 1 & 1 & 1 & 2 & 1\\
    2 & 2 & 1 & x_3 & 1 & 1 & 2 & d\\
    3 & 1 & 1 & 1 & x_4 & 2 & 2 & e\\
    1 & 3 & 1 & 1 & 2 & x_5 & 2 & f\\
    1 & 1 & 2 & 2 & 2 & 2 & x_u & 1\\
    a & b & 1 & d & e & f & 1 & x_v
\end{bmatrix},
\]
The result follows since $\langle \minors_3(M'')\rangle=\langle 1 \rangle$.
\end{proof}

\begin{lemma}
    $\Gm$ is in ${\sf Forb}_2$.
\end{lemma}
\begin{proof}
Let $G$ be a graph having $\Gm$ as induced subgraph, and let $M$ be the submatrix of $D(G,X_G)$ associated with the vertices of $\Gm$ that is shown in Figure~\ref{figure:gm} and where $d_G(v_0,v_4)=y_0$.

\begin{figure}[h!]
\begin{minipage}[c]{0.4\textwidth}
$
\begin{bmatrix}
    x_0 & 2 & 2 & 2 & y_0 & 1\\
    2 & x_1 & 2 & 2 & 1 & 1\\
    2 & 2 & x_2 & 1 & 1 & 1\\
    2 & 2 & 1 & x_3 & 1 & 1\\
    y_0 & 1 & 1 & 1 & x_4 & 2\\
    1 & 1 & 1 & 1 & 2 & x_5
\end{bmatrix},
$
\end{minipage}
\begin{minipage}[c]{0.4\textwidth}
\begin{tikzpicture}[scale=1]
	\tikzstyle{every node}=[minimum width=0pt, inner sep=2pt, circle]
    \draw (60:1) node (v0) [draw] {\small 0};
	\draw (120:1) node (v1) [draw] {\small 1};
	\draw (180:1) node (v4) [draw] {\small 4};
	\draw (240:1) node (v3) [draw] {\small 3};
	\draw (300:1) node (v2) [draw] {\small 2};
    \draw (0:1) node (v5) [draw] {\small 5};
    \Edge[](v0)(v5);
    \Edge[](v1)(v4);
    \Edge[](v1)(v5);
    \Edge[](v2)(v3);
    \Edge[](v2)(v4);
    \Edge[](v2)(v5);
    \Edge[](v3)(v4);
    \Edge[](v3)(v5);
\end{tikzpicture}
\end{minipage}
\caption{Submatrix associated with $\Gm$.}
\label{figure:gm}
\end{figure}

We have that $\det(M[\{0,3,4\},\{1,2,5\}])=3$ and $\det(M[\{0,1,2\},\{3,4,5\}])=1-y_0$.
Since $y_0\in\{2,3\}$, then $\langle 1\rangle=\langle 3, 1-y_0\rangle\subseteq I_3(G,X_G)$, from which follows that $G\in{\sf Forb}_2$.
\end{proof}

\begin{lemma}
    $\Gq$ is in ${\sf Forb}_2$.
\end{lemma}
\begin{proof}
Let $G$ be a graph having $\Gq$ as induced subgraph, and let $M$ be the submatrix of $D(G,X_G)$ associated with the vertices of $\Gq$ that is shown in Figure~\ref{figure:gq} and where $d_G(v_0,v_2)=y_3$, $d_G(v_0,v_3)=y_2$, $d_G(v_1,v_2)=y_1$ and $d_G(v_1,v_3)=y_0$.

\begin{figure}[h!]
\begin{minipage}[c]{0.4\textwidth}
$
\begin{bmatrix}
    x_0 & 1 & y_3 & y_2 & 2 & 2 & 1\\
    1 & x_1 & y_1 & y_0 & 2 & 2 & 1\\
    y_3 & y_1 & x_2 & 2 & 1 & 1 & 2\\
    y_2 & y_0 & 2 & x_3 & 1 & 1 & 2\\
    2 & 2 & 1 & 1 & x_4 & 2 & 1\\
    2 & 2 & 1 & 1 & 2 & x_5 & 1\\
    1 & 1 & 2 & 2 & 1 & 1 & x_6
\end{bmatrix},
$
\end{minipage}
\begin{minipage}[c]{0.4\textwidth}
\begin{tikzpicture}[scale=1]
	\tikzstyle{every node}=[minimum width=0pt, inner sep=2pt, circle]
    \draw (0:1) node (v0) [draw] {\small 0};
	\draw (360/7:1) node (v1) [draw] {\small 1};
	\draw (2*360/7:1) node (v4) [draw] {\small 4};
	\draw (3*360/7:1) node (v3) [draw] {\small 3};
	\draw (4*360/7:1) node (v2) [draw] {\small 2};
    \draw (5*360/7:1) node (v5) [draw] {\small 5};
    \draw (6*360/7:1) node (v6) [draw] {\small 6};
    \Edge[](v0)(v1);
    \Edge[](v0)(v6);
    \Edge[](v1)(v6);
    \Edge[](v2)(v4);
    \Edge[](v2)(v5);
    \Edge[](v3)(v4);
    \Edge[](v3)(v5);
    \Edge[](v4)(v6);
    \Edge[](v5)(v6);
\end{tikzpicture}
\end{minipage}
\caption{Submatrix associated with $\Gq$.}
\label{figure:gq}
\end{figure}

Let
\begin{eqnarray*}
I=\{
x_0 x_1 + x_0 + x_1,
x_0 y_0 + 2 x_0 + y_0 + y_2 + 1,\\
x_0 y_1 + 2 x_0 + y_1 + y_3 + 1,
x_1 y_2 + 2 x_1 + y_0 + y_2 + 1,\\
x_1 y_3 + 2 x_1 + y_1 + y_3 + 1,
x_2 + y_1 y_3 + 2 y_1 + 2 y_3 + 2,\\
x_3 + y_0 y_2 + 2 y_0 + 2 y_2 + 2,
x_4 + 1,
x_5 + 1,
x_6 + 1
\},
\end{eqnarray*}
and
\begin{eqnarray*}
J=\{
y_0 y_1 + 2 y_0 + 2 y_1 + 1,
y_0 y_3 + 2 y_0 + 2 y_3 + 1,\\
y_1 y_2 + 2 y_1 + 2 y_2 + 1,
y_2 y_3 + 2 y_2 + 2 y_3 + 1,
3
\}.
\end{eqnarray*}
It can be verified that the lexicographic Gr\"obner basis of the ideal $\langle \minors_3(M)\rangle$ is equal to $\langle I\cup J\rangle$.
Let ${\bf d}$ be a vector such that the $i$-th entry is taken over a possible values of $y_i$, that is, $d_i\in \{2,3\}$ for $i\in\{0,\dots,3\}$.
Since $\gcd(J|_{\bf y = d})=1$, for any possible vector ${\bf d}$, then $\langle \minors_3(M)\rangle=\langle 1 \rangle$.
Which implies $I_3(G,X_G)$ is trivial.
\end{proof}

Therefore by previous lemmas we have graphs in $\Lambda_2$ are $F$-free.
Now, we are going to prove that graphs in $\Lambda_2$ are {\sf odd-holes}-free.

\begin{lemma}\label{lemma:oddholes forbidden}
    Let $n\geq 3$.
    Then, $\Phi(C_{2n+1})\geq 3$ and no connected induced subgraph $H$ of $G$ has $\Phi(H)=3$.
\end{lemma}
\begin{proof}
    When $n= 3$, we have
    \[
    \det(D(C_7,X_{C_7})[\{0,1,2\},\{4,5,6\}])=2
    \]
    and
    \[
    \det(D(C_7,X_{C_7})[\{1,2,4\},\{3,5,6\}])=5.
    \]
    Since $\gcd(2,5)=1$, then $I(C_7,X_{C_{7}})$ is trivial.
    When $n\geq 4$, we have the following. Let $C_{2n+1}$ be the cycle with vertex set $V(C_{2n})=\{v_i \; : \; i\in\{0,\dots,2n+1\}\}$ and edge set $E(C_{2n+1})=\{v_{i-1} v_{i} \; : \; i\in[2n]\}\cup\{v_{2n},v_0\}$.
    Consider the submatrix $D(C_{2n+1},X_{C_{2n+1}})[\{0,1,2\},\{n-1,n,n+1\}]$:
    \[
    M=
    \begin{bmatrix}
     n-1 & n & n\\
     n-2 & n-1 & n\\
     n-3 & n-2 & n-1\\
    \end{bmatrix}
    \]
    Since $\det(M)=-1$, it follows that $I(C_{2n+1},X_{C_{2n+1}})$ is trivial.
    Finally, any connected induced subgraph of $C_{2n+1}$ is a path.
    In \cite[Theorem 3]{HW}, it was proved that the third invariant factor of the Smith normal form of the distance matrix of a tree is 2.
    Therefore by Corollary~\ref{coro:eval1}, $\Phi(P_k)\leq 2$.
\end{proof}

\begin{lemma}
    For $n\geq 3$, $C_{2n+1}$ is in ${\sf Forb}_2$.
\end{lemma}

\begin{proof}
It only remains to prove that for $n\geq3$ if $G$ is a graph having $C_{2n+1}$ as induced subgraph, then $I_3(G,X_G)$ is trivial.
First note that case $n=3$ is true.

\begin{claim}\label{claim:LemmaInicial:ciclos}
    $\Gp$ is minimal in ${\sf Forb}_2$.
\end{claim}
\begin{proof}
Let $G$ be a graph having $\Gp$ as induced subgraph, and let $M$ be the submatrix of $D(G,X_G)$ associated with the vertices of $\Gp$, that is

\begin{minipage}[c]{0.5\textwidth}
$
\begin{bmatrix}
    x_0 & y_6 & y_5 & 2 & 2 & 1 & 1\\
    y_6 & x_1 & 1 & 2 & 1 & y_4 & 2\\
    y_5 & 1 & x_2 & 1 & 2 & 2 & y_3\\
    2 & 2 & 1 & x_3 & y_2 & 1 & y_1\\
    2 & 1 & 2 & y_2 & x_4 & y_0  & 1\\
    1 & y_4 & 2 & 1 & y_0 & x_5 & 2\\
    1 & 2 & y_3 & y_1 & 1 & 2 & x_6
\end{bmatrix},
$
\end{minipage}
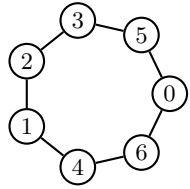
\begin{minipage}[c]{0.5\textwidth}
\begin{tikzpicture}[scale=1]
	\tikzstyle{every node}=[minimum width=0pt, inner sep=2pt, circle]
    \draw (0:1) node (v0) [draw] {\small 0};
	\draw (360/7:1) node (v1) [draw] {\small 5};
	\draw (2*360/7:1) node (v2) [draw] {\small 3};
	\draw (3*360/7:1) node (v3) [draw] {\small 2};
	\draw (4*360/7:1) node (v4) [draw] {\small 1};
    \draw (5*360/7:1) node (v5) [draw] {\small 4};
    \draw (6*360/7:1) node (v6) [draw] {\small 6};
    \Edge[](v0)(v1);
    \Edge[](v1)(v2);
    \Edge[](v2)(v3);
    \Edge[](v3)(v4);
    \Edge[](v4)(v5);
    \Edge[](v6)(v5);
    \Edge[](v0)(v6);
\end{tikzpicture}
\end{minipage}

where $d_G(v_0,v_1)=y_6$, $d_G(v_0,v_2)=y_5$, $d_G(v_1,v_5)=y_4$, $d_G(v_2,v_6)=y_3$, $d_G(v_3,v_4)=y_2$, $d_G(v_3,v_6)=y_1$ and $d_G(v_4,v_5)=y_0$.
We have that $\det(M[\{3,4,5\},\{0,1,2\}]) = 3-2y_4$ and $\det(M[\{4,5,6\},\{0,1,2\}]) = 2y_3 y_4-y_3-2y_4-2$.
Since $y_3,y_4\in\{2,3\}$, then $\langle 1\rangle=\langle 3-2y_4, 2y_3 y_4-y_3-2y_4-2\rangle\subseteq I_3(G,X_G)$, from which follows that $G\in{\sf Forb}_2$.
\end{proof}

We are going to proced by contraposition.
Assume $n\geq 4$ is the first odd positive integer where there exists a graph $G$ with $C_{2n+1}$ as induced subgraph such that $I_3(G,X_G)$ is not trivial.
Thus, for $k\in\{7,\dots, 2n-1\}$, if a graph $H$ has $C_k$ as induced subgraph, then $I_3(H,X_G)$ is trivial.

For simplicity we are going to denote $C_{2n+1}$ as $C$.
There are two cases.
\begin{enumerate}
\item when $d_C(u,v)=d_G(u,v)$ for each pair $u,v\in V(C)$, and
\item there is at least a pair $u,v\in V(C)$ such that $d_G(u,v)<d_C(u,v)$.
\end{enumerate}

In case (1), it follows by Lemma~\ref{lemma:oddholes forbidden} that $I_3(C,X_C)$ is trivial, and then by Lemma~\ref{lemma:inducemonotone}, $I_3(G,X_G)$ is also trivial; which is not possible.

For case (2), take $u,v\in V(C)$ such that $d_G(u,v)<d_C(u,v)$, and $d_G(u,v)$ is minimum over all pairs in $C$.
Let $v_0, \dots, v_{2n+1}$ denote the vertices in $V(C)$ such that $v_k$ is adjacent only with $v_{k-1}$ and $v_{k+1}$.
We abuse of notation to refer $v_k$ to $v_{k\mod 2n+1}$.

\begin{claim}\label{claim:Lemmaunomenosuno:ciclos}
  If $w\in V(G)\setminus V(C)$ is adjacent with $v_k\in V(C)$, then $w$ is not adjacent with neither $v_{k-1}$ nor $v_{k+1}$.
\end{claim}
\begin{proof}
Suppose $w$ is adjacent with $v_{k+1}$.
There are two cases: either $w$ is adjacent with $v_{k+2}$ or not.
Consider the first case.
Vertex $w$ cannot be adjacent with $v_{k-1}$ nor $v_{k+3}$, because otherwise the vertices $w,v_k, v_{k+1}, v_{k+2}, v_{k+3}$ or the vertices $w, v_{k-1}, v_k, v_{k+1}, v_{k+2}$ would induce a \Gd, and by Proposition~\ref{proposition:forbiddensimplecases}, $I_3(G,X_G)$ would be trivial.
Therefore, the vertices $w, v_{k-1}, v_k, v_{k+1}, v_{k+2}, v_{k+3}$ will induce a \Gl.
Thus by Lemma~\ref{lemma:forbidden_co-twin-house}, $I_3(G,X_G)$ would be trivial, which is impossible.
Now consider the case when $w$ is not adjacent with $v_{k+2}$.
There are two possible scenarios: either $w$ is adjacent with $v_{k-1}$ or not.
In the first one, we also have to consider whether $w$ is adjacent with $v_{k-2}$ or not.
If $w$ is adjacent with $v_{k-2}$, then the vertices $w, v_{k-2}, v_{k-1}, v_k, v_{k+1}$ induce a \Gd, which is not possible.
If $w$ is not adjacent with $v_{k-2}$, then the vertices $w, v_{k-2}, v_{k-1}, v_k, v_{k+1}$ will induce a \Gl.
Therefore, $w$ is not adjacent with $v_{k-1}$.
And thus, the vertices $w, v_{k-1}, v_k, v_{k+1}, v_{k+2}$ induce a \Ga; which is not possible since Lemma~\ref{lemma:bullisforbidden} implies $I_3(G,X_G)$ must be trivial.
Then, $w$ is not adjacent with $v_{k+1}$.
Analogously, we can obtain that $w$ is not adjacent with $v_{k-1}$.
\end{proof}

\begin{claim}\label{claim:Lemmatresmenostres:ciclos}
  If $w\in V(G)\setminus V(C)$ is adjacent with $v_k\in V(C)$, then $w$ is not adjacent with neither $v_{k-3}$ nor $v_{k+3}$.
\end{claim}
\begin{proof}
Suppose $w$ is adjacent with $v_{k+3}$.
By Claim~\ref{claim:Lemmaunomenosuno:ciclos}, $w$ is not adjacent with $v_{k+1}$, $v_{k+2}$ nor $v_{k+4}$.
Then, the vertices $w, v_k, v_{k+1}, v_{k+2}, v_{k+3}, v_{k+4}$ induce a \Gg.
Thus by Lemma~\ref{Lemma:5-panforbidden}, $I_3(G,X_G)$ is  trivial, which is impossible.
\end{proof}

Next claim follows since $n$ is the first integer greater or equal than 4 for which there exists a graph $G$ with $C_{2n+1}$ as induced subgraph such that $I_3(G,X_G)$ is not trivial.

\begin{claim}\label{claim:Lemmaimparesmenos:ciclos}
If $w\in V(G)\setminus V(C)$ is adjacent with $v_k\in V(C)$, then $w$ is not adjacent with $v_{k-l}$ nor $v_{k+l}$, where $l\in\{5,7,\dots, 2n-1\}$.
\end{claim}
\begin{proof}
Assume $w$ is adjacent with $v_{k+5}$.
By applying Claims \ref{claim:Lemmaunomenosuno:ciclos} and \ref{claim:Lemmatresmenostres:ciclos} to $v_k$ and $v_{k+5}$, we have $w$ is not adjacent with $v_{k+1}$, $v_{k+2}$, $v_{k+3}$ and  $v_{k+4}$.
Then $w$, $v_k$, $v_{k+1}$, $v_{k+2}$, $v_{k+3}$, $v_{k+4}$, $v_{k+5}$ induce a $C_7$ which has at least 3 trivial distance ideals.
That is not possible.
The results follows by applying recursively the argument.
\end{proof}

Claims \ref{claim:Lemmaunomenosuno:ciclos}, \ref{claim:Lemmatresmenostres:ciclos} and \ref{claim:Lemmaimparesmenos:ciclos} imply that if a vertex $w\in V(G)\setminus V(C)$ is adjacent with $v_k\in V(C)$, then $w$ cannot be adjacent with $v_{k-1}$, $v_{k-3}$, \dots, $v_{-k+1}$ neither with $v_{k+1}$, $v_{k+3}$, \dots, $v_{3k-1}$.
That is not other thing that next result.

\begin{claim}\label{claim:siwvkwalomasuno}
If $w\in V(G)\setminus V(C)$ is adjacent with $v_k\in V(C)$, then $w$ is not adjacent with any other vertex in $V(C)\setminus\{v_k\}$.
\end{claim}

On the other hand, since $u$ and $v$ are two vertices in $V(C)$ such that $d_G(u,v) < d_C(u,v)$, then there exists a path $P=u, u_1, \dots, u_{l-1},v$ of minimum length $l=d_G(u,v)$.
We might assume $u_k\notin V(C)$ for each $k\in\{1,\dots,l-1\}$.
For simplicity, we also relabel the vertices of $C$ such that $u=v_0$ and $v=v_m$.
By Claim~\ref{claim:siwvkwalomasuno}, for $k\in\{1,\dots,l-1\}$, vertex $u_k$ is adjacent with at most one vertex in $V(C)$.
In particular, the length $l$ cannot be 2.
See Figure~\ref{figure:CP}.

\begin{figure}
\begin{tikzpicture}[scale=1]
	\tikzstyle{every node}=[minimum width=0pt, inner sep=2pt, circle]
    \draw (140:2.8) node (v_1) [draw,fill=white] {\tiny $v_{-1}$};
    \draw (160:3) node (v0) [draw,fill=white] {\tiny $v_0$};
    \draw (180:2.8) node (v1) [draw,fill=white] {\tiny $v_1$};
    
    \draw (40:2.8) node (vm1) [draw,fill=white] {\tiny $v_{m+1}$};
    \draw (20:3) node (vm) [draw,fill=white] {\tiny $v_m$};
    \draw (0:2.8) node (vm_1) [draw,fill=white] {\tiny $v_{m-1}$};
    
    \draw (-2,0.5) node (u1) [draw,fill=white] {\tiny $u_{1}$};
    \draw (0,0.5) node (uk) [draw,fill=white] {\tiny $u_{k}$};
    \draw (1.8,0.5) node (ul) [draw,fill=white] {\tiny $u_{l-1}$};
    
    \draw (-1,-1.5) node (vj_1) [draw,fill=white] {\tiny $v_{j-1}$};
    \draw (0,-1.5) node (vj) [draw,fill=white] {\tiny $v_{j}$};
    \draw (1,-1.5) node (vj1) [draw,fill=white] {\tiny $v_{j+1}$};
    
    \draw (0,-.5) node (C1) [] {$C^1$};
    \draw (0,1.3) node (C2) [] {$C^2$};
    
    \draw (v_1) -- (v0) -- (v1);
    \draw (vm_1) -- (vm) -- (vm1);
    \draw (vj_1) -- (vj) -- (vj1);
    
    \draw[dotted] (v1) edge[bend right] (vj_1);
    \draw[dotted] (vm_1) edge[bend left] (vj1);
    \draw[dotted] (u1) -- (uk) -- (ul);
    \draw[dotted] (vm1) edge (v_1);
    
    \draw (v0) -- (u1);
    \draw (ul) -- (vm);
\end{tikzpicture}
\caption{A drawing of the induced subgraph by $C$ and $P$.}
\label{figure:CP}
\end{figure}
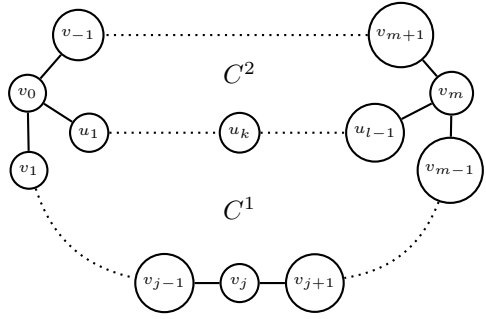

Let $C^1$ be the induced subgraph obtained by the vertices $\{v_0, v_1, \dots, v_{m-1}, v_m\}\cup V(P)$, and let $C^2$ be the induced subgraph obteined by $\{v_0, v_{-1}, \dots, v_{m+1}, v_m\}\cup V(P)$.
Note that either $C^1$ or $C^2$ has an odd number of vertices.
Let us assume without loss of generality $C^1$ is the side with an odd number of vertices.
In $C^1$ we want to find an induced cycle with an odd number of vertices.
In fact, if $C^1$ is already an induced cycle of odd length, then this cannot be of length 3 or 5.
Otherwise $P= v_0 u_1 v_m$ consists of three vertices, this is because $P$ have less vertices than the path $v_0, v_1, \dots, v_m$.
Which is impossible since by Claim~\ref{claim:siwvkwalomasuno}, every vertex in $V(P)\setminus\{v_0,v_m\}$ is adjacent with at most one vertex in $C$.
Therefore $C^1$ is an odd induced cycle of length less than $2n+1$ and greater or equal than $7$.
Which is also impossible, because otherwise $I_3(G,X_G)$ is  trivial.

So, suppose there is at least one vertex in $V(P)\setminus\{v_0,v_m\}$ that is adjacent with a vertex in $P'= v_1 v_2 \cdots v_{m-1}$.
Let $u_k\in P$ be the first vertex from $v_0$ that is adjacent with a vertex $v_j\in P'$.
Note that $u_k$ cannot be adjacent with a vertex in $V(P)\setminus\{u_{k-1},u_k, u_{k+1}\}$, otherwise $P$ would not be of minimum length.
Also we have the next result which can be proved analogously to Claim~\ref{claim:Lemmaunomenosuno:ciclos}.
\begin{claim}\label{claim:VinCadjacentwithP}
  Let $v_j\in V(C^1)\setminus V(P)$.
  If $v_j$ is adjacent with $u_k\in V(P)$, then $v_j$ is not adjacent with $u_{k-1}$ nor $u_{k+1}$.
\end{claim}

Thus, Claims \ref{claim:Lemmatresmenostres:ciclos} and \ref{claim:VinCadjacentwithP} imply that $C^1$ does not contain a $C_3$ as induced subgraph.

The edge $u_k v_j$ divides $C^1$ in two induced subgraphs: an induced cycle $C^3$ with vertex set $\{ v_0, v_1, \dots, v_j, u_k, u_{k-1}, \dots, u_1\}$, and the induced subgraph obtained by the vertices $u_k, v_j, v_{j+1}, \dots, v_m, u_{l-1}, \dots, u_{k+1}$.
One of these induced subgraphs have an odd number of vertices.

If the cycle is the induced subgraph of odd length, then we are almost done.
This cycle $C^3$ cannot have length 3 or greater than 7, since otherwise $G$ would have three trivial distance ideals.
We only need to consider when $C^3$ has length 5.
But there are only two possibilities: either $V(C^3)$ is equal to $v_0, v_1, u_3, u_2, u_1$ or $v_0, v_1, v_2, u_2, u_1$.
However, in both cases the vertices $v_0, v_1, u_3, u_2, u_1, u_4$ and  $v_{-1}, v_0, v_1, v_2, u_2, u_1$ would induce a \Gg.
Which is not possible since otherwise $G$ would have three trivial distance ideals.
Then the vertices $u_k, v_j, v_{j+1}, \dots, v_m, u_{l-1}, \dots, u_{k+1}$ induce a subgraph with an odd number of vertices.

Assume that, and take the first vertex, say $u_r$, in $P$ from $u_k$ that is adjacent with a vertex, say $v_{s}$, in $P'\setminus\{v_0, v_1, \dots, v_{j-1}\}$.
If such vertex does not exist we are done with a similar argument than in previous paragraph.
Therefore assume, $v_{s}$ is adjacent with $u_{r}$ such that $r>k$ and $s\geq j$.
Again from the edge $u_r v_s$ we obtain two induced subgraphs where one has odd length: let $C^4$ be the subgraph induced cycle by the vertices $u_k, u_{k+1}, \dots, u_r, v_s, v_{s-1}, \dots, v_j$, and the other subgraph is the induced by $u_r, v_s, v_{s+1}, \dots, v_m, u_{l-1}, \dots, u_{r+1}$.
If $C^4$ has odd length, again we only need to consider when it has length 5.
There are three possible cases: either $r=k+1$, $r=k+2$ or $r=k+3$.
See Figure~\ref{figure:CasesToConsider}.

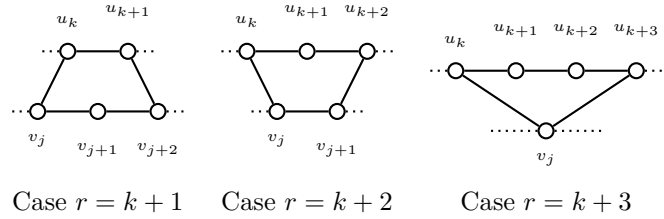
\begin{figure}[h!]
\begin{tabular}{ccc}
\begin{tikzpicture}[scale=0.8]
	\tikzstyle{every node}=[minimum width=0pt, inner sep=2pt, circle]
    \draw (-0.5,0) node (v0) [draw,fill=white,label=above:{\tiny $u_k$}] {};
    \draw (0.5,0) node (v1) [draw,fill=white,label=above:{\tiny $u_{k+1}$}] {};
    \draw (-1,-1) node (v2) [draw,fill=white,label=below:{\tiny $v_j$}] {};
    \draw (0,-1) node (v3) [draw,fill=white,label=below:{\tiny $v_{j+1}$}] {};
    \draw (1,-1) node (v4) [draw,fill=white,label=below:{\tiny $v_{j+2}$}] {};
    \draw (v2) -- (v0) -- (v1) -- (v4) -- (v3) -- (v2);
    \draw[dotted] (v0) edge (-1,0);
    \draw[dotted] (v1) edge (1,0);
    \draw[dotted] (v2) edge (-1.5,-1);
    \draw[dotted] (v4) edge (1.5,-1);
\end{tikzpicture}
&
\begin{tikzpicture}[scale=0.8]
	\tikzstyle{every node}=[minimum width=0pt, inner sep=2pt, circle]
    \draw (-1,0) node (v0) [draw,fill=white,label=above:{\tiny $u_k$}] {};
    \draw (0,0) node (v1) [draw,fill=white,label=above:{\tiny $u_{k+1}$}] {};
    \draw (-0.5,-1) node (v2) [draw,fill=white,label=below:{\tiny $v_j$}] {};
    \draw (0.5,-1) node (v3) [draw,fill=white,label=below:{\tiny $v_{j+1}$}] {};
    \draw (1,0) node (v4) [draw,fill=white,label=above:{\tiny $u_{k+2}$}] {};
    \draw (v0) -- (v1) -- (v4) -- (v3) -- (v2) -- (v0);
    \draw[dotted] (v0) edge (-1.5,0);
    \draw[dotted] (v4) edge (1.5,0);
    \draw[dotted] (v2) edge (-1,-1);
    \draw[dotted] (v3) edge (1,-1);
\end{tikzpicture}
&
\begin{tikzpicture}[scale=0.8]
	\tikzstyle{every node}=[minimum width=0pt, inner sep=2pt, circle]
    \draw (-1,0) node (v0) [draw,fill=white,label=above:{\tiny $u_k$}] {};
    \draw (0,0) node (v1) [draw,fill=white,label=above:{\tiny $u_{k+1}$}] {};
    \draw (0.5,-1) node (v2) [draw,fill=white,label=below:{\tiny $v_j$}] {};
    \draw (2,0) node (v3) [draw,fill=white,label=above:{\tiny $u_{k+3}$}] {};
    \draw (1,0) node (v4) [draw,fill=white,label=above:{\tiny $u_{k+2}$}] {};
    \draw (v0) -- (v1) -- (v4) -- (v3) -- (v2) -- (v0);
    \draw[dotted] (v0) edge (-1.5,0);
    \draw[dotted] (v3) edge (2.5,0);
    \draw[dotted] (v2) edge (-0.5,-1);
    \draw[dotted] (v2) edge (1.5,-1);
\end{tikzpicture}
\\
Case $r=k+1$ & Case $r=k+2$ & Case $r=k+3$\\
\end{tabular}
\caption{Cases when $C^4$ contains 5 vertices.}
\label{figure:CasesToConsider}
\end{figure}
In case $r=k+1$, $u_k$ is different from $v_0$ since $u_k$ is already adjacent with $v_j$.
Consider $u_{k-1}$.
By Claim~\ref{claim:VinCadjacentwithP}, $v_j$ is not adjacent with $u_{k-1}$.
However, it is possible that at most one of the vertices $v_{j+1}$ or $v_{j+1}$ is adjacent with $u_{k-1}$.
In any of the three possibilities we would obtain a $\Gg$ or $\Gk$ as induced subgraph og $G$, which is impossible.

Similarly, in cases $r=k+2$ and $r=k+3$, by considering vertex $v_{j-1}$, we will obtain that $C^4$ and $v_{j-1}$ induce a $\Gg$ or $\Gk$.
Which is impossible.
Therefore, $C^4$ is of even length.

Continuing inductively with the next vertex in $P'$ after $v_r$ that is adjacent with a vertex $v_t$ with $t\geq s$, we will obtain that there is an odd cycle in $C^1$ of length greater or equal than 5.
And we will get a contradiction.
\end{proof}

Thus by previous lemmas we have our main result.
\begin{theorem}
    Graphs in $\Lambda_2$ are $\{F,\text{\sf odd-holes}\}$-free.
\end{theorem}

The remaining difficult question is: are all $\{F,\text{odd-holes}\}$-free graphs in $\Lambda_2$?


\section*{Acknowledgement}
The author would like to thank Jephian C.-H. Lin for helpful comments.
This research was partially supported by SNI and CONACyT.

\end{document}